\documentclass[11pt]{article}

\usepackage[latin1]{inputenc}
\usepackage[english]{babel}

\usepackage{amsmath,amsthm,amssymb}
\usepackage{epsfig}
\usepackage{graphics}
\usepackage{graphicx}
\usepackage{booktabs}
\usepackage{ctable}
\usepackage{subfigure}
\usepackage{float}
\usepackage{layout}
\usepackage{fancyhdr}
\usepackage{colortbl}  
\usepackage{color}     
\usepackage{booktabs}  
\usepackage{longtable} 
\usepackage{lscape}    
\usepackage{rotating}  
\usepackage{multirow}  
\usepackage{makeidx}   
\usepackage{varioref}
\usepackage{xr}
\usepackage{calc}
\usepackage{eurosym}   

\input amssym.def
\input amssym.tex

\makeindex 

\def \frac#1#2{{#1\over #2}}
\def\1{{\bf 1}}
\def\0{{\bf 0}}
\def\2{{\bf 2}}

\def\x{{\bf x}}
\def\y{{\bf y}}

\def\v{{\bf v}}
\def\w{{\bf w}}
\def\b{{\bf b}}

\def\cfrac#1#2{{#1\over #2}}

\newtheorem{defi}{Definition}[section]
\newtheorem{risu}{Result}[section]
\newtheorem{teo}{Theorem}[section]
\newtheorem{prop}{Proposition}[section]

\newtheorem{rem}{Remark}[section]

\newcommand{\be}{\begin{enumerate}}
\newcommand{\ee}{\end{enumerate}}
\newcommand{\bi}{\begin{itemize}}
\newcommand{\ei}{\end{itemize}}
\newcommand{\beq}{\begin{equation}}
\newcommand{\eeq}{\end{equation}}

\numberwithin{equation}{section} 

\begin{document}

\baselineskip=14 pt

\title{$K-$Fibonacci sequences and minimal winning quota\\ in Parsimonious games\footnote{A very preliminary version of the paper has been presented at the 2013 Workshop of the Central European Program in Economic Theory, which took place in Udine (20-21 June) and may be found in CEPET working papers \cite{PrePla13}.}}

\author{Flavio Pressacco\textsuperscript{a}, Giacomo Plazzotta\textsuperscript{b} and Laura Ziani\textsuperscript{c}}

\date{}

\maketitle

\noindent\textsuperscript{a} Dept. of Economics and Statistics D.I.E.S., Udine University, Italy\\ 
\textsuperscript{b} Imperial College London, UK\\
\textsuperscript{c} Dept. of Economics and Statistics D.I.E.S., Udine University, Italy 

\begin{abstract}
Parsimonious games are a subset of constant sum homogeneous weighted majority games unequivocally described by their free type representation vector. We show that the minimal winning quota of parsimonious games satisfies a second order, linear, homogeneous, finite difference equation with nonconstant coefficients except for uniform games. We provide the solution of such an equation which may be thought as the generalized version of the polynomial expansion of a proper $k-$Fibonacci sequence.
In addition we show that the minimal winning quota is a symmetric function of the representation vector; exploiting this property it is straightforward to prove that twin Parsimonious games, i.e. a couple of games whose free type representations are each other symmetric, share the same minimal winning quota.
\end{abstract}

\renewcommand{\abstractname}{Keywords}
\begin{abstract}
Homogeneous weighted majority games,
parsimonious games, minimal winning quota, $k-$Fibonacci sequence, Fibonacci polynomials.
\end{abstract}

\renewcommand{\abstractname}{Acknowledgements}
\begin{abstract}
 We acknowledge financial support of MedioCredito Friuli Venezia Friuli through the ``Bonaldo Stringher'' Laboratory of Finance, Department of Finance, University of Udine.
\end{abstract}

\section{Introduction}

In this paper we treat the problem of finding a closed formula for the minimal winning quota of parsimonious games (henceforth $P$ games) as a function of the free type representation of the game.

$P$ games have been introduced in a vintage paper by Isbell (\cite{Isb56}, 1956) as the subset of constant sum homogeneous weighted majority games characterized by the parsimony property to have, for any given number $n$ of non dummy players in the game, the smallest number, i.e. exactly $n$, of minimal winning coalitions.

It turns out that, in any $n$ person $P$ game, there are $h$ $(2\leq h\leq n-2)$ types of players; in the minimal homogeneous representation of the game, type $t$ players share the same type weight $w_t$. Conversely a $P$ game is fully and unequivocally described by its type representation, the ordered (according to an increasing weight convention) $h$ dimension vector $\x=(x_1,\ldots,x_t,\ldots,x_h)$, whose component $x_t$ is the (obviously positive integer) number of type $t$ players in the game. Besides lower bounds on $x_1$ and $x_{h-1}$, this vector should satisfy the binding constraint $x_h=1$. Then, what really matters is the free type version $_f\x=(x_1,\ldots,x_t,\ldots,x_{h-1})$ of the representation, which is obtained by deleting the last component $x_h$ of $\x$. In particular, we shall see that an important role in what follows is played by the subset of $k-$uniform $P$ games (henceforth $k-UP$ games), those with uniform free type representation at the level $x_t=k=(n-1)/(h-1)$.

Elsewhere (\cite{PPZ1}, sect. 2, formula 2.1) we put in evidence that starting from the standard initial conditions on the weight of the groups labelled 0 (dummy players) and 1 (non dummy with smallest weight), a simple rule gives, for any $P$ game, the sequence of type weights for all but the top player as the recursive solution of a second order linear homogeneous finite difference equation, with nonconstant coefficients except for the subset of $k-UP$ games.

Moreover, keeping account of the rule governing the weight of the top player and of the fact that in any $P$ game the coalition made by the top player and one of the last but top is minimal winning, we show here that also the minimal winning quota $q$ is, for any $h$, the solution of a finite difference equation with exactly the same structure and the same initial conditions of the weight equation and hence with non constant coefficients except for the $k-UP$ games.
In particular, for the subset of $k-UP$ games the corresponding finite difference equations have constant coefficient $k$.

Looking at the wide literature (see citations in \cite{Wen97} and in \cite{FalPla07}) on properties and solutions of second order linear recurrence equations you can check that special attention has been devoted for $k$ positive integer to $k-$Fibonacci equations, i.e. equations of the form $F_{n}(k)=k\cdot F_{n-1}(k)+F_{n-2}(k)$ with initial conditions $F_{0}(k)=0$ and $F_{1}(k)=1$, whose solution is known as the $k-$Fibonacci sequence.

It was a lucky surprise to realize that both the structure and the initial conditions of the $k-$Fibonacci equations coincide with those of the weight problem and of the minimal winning quota problem of $k-UP$ games.

This makes clear that the problem of finding the sequence of weights (or respectively the sequence of minimal winning quotas) of $k-UP$ games is isomorphic to the one of solving the corresponding constant coefficients $k-$Fibonacci equation.

We recall that closed form of this solution may be expressed either by the Binet's formula (see \cite{FalPla07} and \cite{StRo06}) or by a polynomial expansion in $k$, whose coefficients may be linked to Pascal triangles (precisely, to the coefficients of the so called 2-Pascal triangles, see \cite{FalPla07} and \cite{FalPla09} or of the ``shallow diagonals'' of the classic Pascal triangle, see for example \cite{AnCaPe11}).
We found useful to introduce here the alternative but perfectly equivalent polynomial representation, $q_h(k)=\sum_{s=0,\ldots,h-1} C^{'}(h,s)\cdot k^s$ whose coefficients are given by the elements $C^{'}(h,s)$ of a conveniently modified Pascal triangle.
Going back now to the minimal winning quota of the general case of $P$ game non $U$, it is clear that this problem too is
isomorphic to the one of solving the corresponding nonconstant coefficients finite difference equation.

This suggested to us that the solution could be expressed by the generalized version (in terms of the free type representation $_f\x$) of the polynomial expansion in $k$ found for the $UP$ (constant coefficients) case. Indeed, the key connection between the particular and the general case comes from the coefficients $C^{'}(h,s)$ of the polynomial expansion, which play the same role in both cases. Precisely it turns out that, for any combination of $h$ types and $s$ factors, the $C^{'}(h,s)$ which multiplied $k^s$ in the polynomial expansion of $q_h(k)$ in the $U$ case, is now the number of different products of $s$ factors chosen (according to proper feasibility rules) from the free type representation vector. After that the minimal winning quota is simply the sum over all $s$ (from 0 to $h-1$) of such products.

Among other things, it implies that the number of addends of the polynomial expansion of the minimal winning quota of any $P$ game with $h$ types is exactly the $h-$th number of the (classical) Fibonacci sequence.

Summing up the first part of the paper, we may say that the minimal winning quota of any $P$ game with $n$ players and $h$ types may be seen as a generalized version of the $h-$th number of the $k-$Fibonacci sequence, with $k=(n-1)/(h-1)$. See examples 2 and 3 in section \ref{sect:9}.

In the second part of the paper we will see that the solution of our finite difference equation is a symmetric function of the free type representation of the game. This paves the way for an alternative proof of the result given elsewhere (see \cite{PPZ2}, sect. 4) on the equality of the minimal winning quota of any couple of twin games, i.e. games whose free type representations are each other symmetric.

The plan of the paper is as follows. A short recall of the
properties of $P$ games, useful for our treatment, is given in
section \ref{sect:2}. Section \ref{sect:3} is devoted to derive
the difference equation describing the behaviour of the minimal
winning quota in not $U$ as well as in $UP$ games. Section
\ref{sect:4} recalls some results on $k-$Fibonacci polynomials
and $k-$Fibonacci sequences. Section \ref{sect:5} gives closed
form solutions of the finite difference equation of the minimal
winning quota for $UP$ games, along with a discussion of the
connection between the polynomial expansion and the
coefficients of the modified Pascal triangle resuming such
coefficients. Section \ref{sect:6} builds a bridge between the
solution of the $U$ and the not $U$ case, and describes the
rule behind the generalized version of the polynomial
expansion, giving the minimal winning quota for the not $U$
games as a function $q_h(_f\x)$ of the free type representation
of the game. In section \ref{sect:7} we show that $q_h$ is a
symmetric function of its argument $(_f\x)$. Section
\ref{sect:8} gives an alternative proof of the equality of the
minimal winning quota of twin games based on such a property of
symmetry; section \ref{sect:9} offers some examples and
conclusions follow in section \ref{sect:10}.

\section{A short recall of the main properties of $P$
games}\label{sect:2}

As usual $N$ is the set of players whose number is supposed
here greater than three\footnote{We will consider here $P$ games
with $n>3$, a necessary condition to have at least two types of players in the game. See \cite{Isb56}, p. 185}, $S$ is any
coalition (subset of players), $v(S)$ is the characteristic
function of the game.

Let us recall that constant sum homogeneous weighted majority games\footnote{At the origins of
game theory, homogeneous weighted majority games (h.w.m.g.)
have been introduced in \cite{VNM47} by Von Neumann-Morgenstern
and have been studied mainly under the constant sum condition.
Subsequent treatments in the absence of the constant sum
condition (with deadlocks) may be found e.g. in \cite{Ost87} by
Ostmann, who gave the proof that any h.w.m.g. (including non
constant sum ones) has a unique minimal homogeneous
representation, and in \cite{Ros87}. Generally speaking, the
homogeneous minimal representation is to be thought in a
broader sense but hereafter the restrictive application
concerning the constant sum case is used.} are
simple ($v(S)=0$ or 1), constant sum ($v(S)+v(\widetilde{S})=1$) $n$ person games, which admit a minimal
homogeneous representation $(q,\w)$ in which
$q=\cfrac{1+w(N)}{2}$ is the minimal winning quota $(v(S)=1
\Leftrightarrow w(S)=\sum_{j\in S}w_j\geq q$), and $\w$ is an
ordered vector of individual weights $(w_j\leq w_{j+1})$ such
that all weights are integer, there are players with weight 1
and for any coalition $S$ minimal winning\footnote{A coalition $S$ is said to be minimal winning if $v(S)=1$ and, for any $T\varsubsetneq S$, $v(T)=0$. The set of minimal winning coalitions is denoted by $WM$.}, $w(S)=q$, which
implies for such coalitions $w(S)-w(\widetilde{S})=1$.

The number of minimal winning coalitions of a $n$ (non dummy
and greater than 3) person constant sum homogeneous weighted majority game
may be either greater or equal, but not lower than $n$ (see \cite{Isb56}, p. 185).
Parsimonious games are the subset of constant sum homogeneous weighted
majority games with exactly $n$ minimal winning coalitions.

Players of a $P$ game may be divided in $h$ $(2\leq h\leq n-2)$
subsets, grouping players with the same individual weight.

An alternative, very important in our treatment, representation
of a $P$ game is the ordered vector
$\x=(x_1,\ldots,x_t,\ldots,x_{h})$, whose component $x_t$ is the
number of players of type $t$ (i.e. with common individual
weight $w_j=w_t$ for any player $j$ of type $t$) in the game.
The group labelling is coherent with the increasing weight convention
that $w_{t-1}< w_t$. The following lower bounds and constraints (\cite{Isb56}, p.
185) hold on $\x$: $x_1$ and $x_{h-1}>1,$ $x_h =1$.

We suggest to call
type representation of a $P$ game the vector $\x$, and free
type representation the vector
$_f\x=(x_1,\ldots,x_t,\ldots,x_{h-1})$ obtained by deleting the
last component of $\x$.

Moreover, we define $k-$uniform $P$ (henceforth $k-UP$) games, the subset of $P$ games, whose free type representation is uniform at level $k$, i.e.:
\begin{equation}\label{Eqn:kUP}
    x_t=k=\cfrac{n-1}{h-1}\quad \text{for any } t=1,\ldots,h-1
\end{equation}

A simple recursion rule (see \cite{PPZ1}, sect. 2, formula 2.1)
gives the type weights of a $P$ game starting from the initial
conditions $w_0=0$ and
$w_1=1$:
\begin{subequations}\label{Eqn:21}
\begin{align}
w_0&=0\label{Eqn21:first}\\
w_1&=1\label{Eqn21:second}\\
w_t&=x_{t-1}\cdot w_{t-1}+ w_{t-2}, \quad \forall 1<t<h\label{Eqn21:third}\\
w_h&=(x_{h-1}-1)\cdot w_{h-1}+w_{h-2}\label{Eqn21:forth}
\end{align}
\end{subequations}

\begin{rem}\label{rem:21new}
We underline that \eqref{Eqn21:third} jointly with the initial conditions \eqref{Eqn21:first} and \eqref{Eqn21:second} describe a second order linear homogeneous finite difference equation, whose recursive solution gives the type weights $w_t$ of all, but the top, players. Clearly, the equation does not have constant coefficients, except in the $k-UP$ case in which (initial conditions unchanged):
\begin{subequations}\label{Eqn:21bis}
\begin{align}
w_0&=0\label{Eqn21bis:first}\\
w_1&=1\label{Eqn21bis:second}\\
w_t&=k\cdot w_{t-1}+ w_{t-2}, \quad \forall 1<t<h\label{Eqn21bis:third}\\
w_h&=(k-1)\cdot w_{h-1}+w_{h-2}\label{Eqn21bis:forth}
\end{align}
\end{subequations}
\end{rem}

\section{The minimal winning quota in $P$ games}\label{sect:3}



Let us consider for any $h\geq 2$ the two sequences $q_h$ and
$w_h$ defined respectively as the minimal winning quota of a
$P$ game with exactly $h$ types of players and the type weight
 of type $h$ players in a $P$ game with more than $h$ types.

\begin{rem}
Hereafter it is convenient to distinguish the expressions of
$w_h$ coming from the previous definition, from the ones,
denoted by $w_h^*$, representing the weight of the top player
in a $P$ game with exactly $h$ types.
\end{rem}

Coherently with such definitions, for any $h$ we have by
\eqref{Eqn21:third}:
\begin{equation}\label{Eqn:31a}
    w_h=x_{h-1}\cdot w_{h-1}+w_{h-2}
\end{equation}
while by \eqref{Eqn21:forth} we know that:
\begin{equation}\label{Eqn:32a}
    w_h^{*}=(x_{h-1}-1)\cdot w_{h-1}+w_{h-2}
\end{equation}

Moreover, keeping account that in any $P$ game the coalition
made by the top player and one of the last but top
(see \cite{PPZ2}, prop. 3.1) is
minimal winning, it is: for any $h\geq 2$:
\begin{equation}\label{Eqn:33a}
    q_h=w_h^{*}+w_{h-1}=x_{h-1}\cdot w_{h-1}+w_{h-2}
\end{equation}

If we could choose the following initial conditions on the
function $q_h$:
\begin{equation}\label{Eqn:34a}
    q_0=w_0=0
\end{equation}
\begin{equation}\label{Eqn:35a}
    q_1=w_1=1
\end{equation}
this choice, jointly with \eqref{Eqn:31a} and \eqref{Eqn:33a},
would imply by immediate induction on $h$:
\begin{equation}\label{Eqn:36a}
    q_h=w_h\quad \text{$\forall h$}
\end{equation}
and by substitution in \eqref{Eqn:33a}:
\begin{equation}\label{Eqn:37a}
    q_h=x_{h-1}\cdot q_{h-1}+q_{h-2}
\end{equation}
and keeping account of the initial conditions on $q_h$ the
following result would hold:

\begin{teo}\label{teo:31a}
The minimal winning quota $q_h$ of any $P$ game with $h\geq 2$
types of players satisfies the second order, linear,
homogeneous, nonconstant coefficients, finite difference
equation \eqref{Eqn:37a} with initial conditions
\eqref{Eqn:34a}, \eqref{Eqn:35a}.
\end{teo}

It remains to justify the initial conditions, which \textit{per
se} are meaningless, as there are no $P$ games with less than
two types (see footnote 2, p. 3); yet this choice
is the unique coherent with the true expressions of the really
substantial initial conditions:
\begin{equation}\label{Eqn:38a}
    q_2=x_1
\end{equation}
\begin{equation}\label{Eqn:39a}
    q_3=1+x_1x_2
\end{equation}
easily derived by direct reasoning as follows.

In a $P$ game
with $h=2$ there are $x_1\geq 3$ players of type 1, with weight
1, and one top player of type 2 with weight $x_1-1$; one of the
minimal winning coalitions is made by one player of type 1 and
the top player, hence by \eqref{Eqn:21} $q_2=x_1$ is the
expression of the minimal winning quota of a $P$ game with
$h=2$ types.

In $P$ games with $h=3$ there are $x_1$ players of
type 1 (type weight $w_1 =1$), $x_2$ players of type 2 (type
weight $w_2=x_1$) and the type 3, top player, with weight
$(x_2-1)\cdot w_2+w_1=1-x_1+x_1x_2$. One of the minimal winning
coalitions is made by one player of type 2 and the top player,
hence the minimal winning quota of $P$ games with $h=3$ types
is $q_3= x_1 +1- x_1+ x_1x_2=1+x_1x_2$. After that, elementary
algebra gives \eqref{Eqn:34a} and \eqref{Eqn:35a} as the unique
solution of the system of \eqref{Eqn:37a}, \eqref{Eqn:38a},
\eqref{Eqn:39a}, and Theorem \ref{teo:31a} has been proved.

\begin{rem}\label{Rem:32old}
In the case of $k-UP$, formula \eqref{Eqn:37a} becomes:
\begin{equation}\label{Eqn:37new}
    q_h=k\cdot q_{h-1}+q_{h-2}
\end{equation}
\end{rem}

\section{Fibonacci polynomials and $k-$Fibonacci sequences}\label{sect:4}

\begin{defi}
The sequence of polynomials in the real variable $x$ defined by the recursive relation:
\begin{equation}\label{Eqn:poly1}
    P_n(x)=x\cdot P_{n-1}(x)+P_{n-2}(x)
\end{equation}
with initial conditions $P_0(x)=0$ and $P_1(x)=1$ for any $x$
is known as the sequence of Fibonacci polynomials (see \cite{HoBi73}.
\end{defi}

\begin{rem}
It is immediate to check that the evaluation of the values $P_n(1)$ of the sequence at $x=1$ gives the sequence of Fibonacci numbers $P_n(1)=F_n$.
\end{rem}
The following generalization of this result has been given:

\begin{defi}
For any given positive integer $k$, the sequence $P_n(k)$ is defined as the sequence of $k-$Fibonacci numbers or shortly $k-$Fibonacci sequence (see \cite{FalPla07}). The notation $P_n(k)=F_n(k)$ may be used.
\end{defi}

\begin{rem}\label{rem:42new}
$P_n(k)$ is the solution of the second order finite difference equation with constant coefficients:
\begin{equation}\label{Eqn:poly2}
    P_n(k)=k\cdot P_{n-1}(k)+P_{n-2}(k)
\end{equation}
with initial conditions $P_0(k)=0$ and $P_1(k)=1$.
\end{rem}

In particular, the $1-$Fibonacci sequence is just the classic
one, the $2-$Fibonacci sequence (0, 1, 2, 5, 12, 29, 70,...) is
known as the Pell sequence, the $3-$Fibonacci
sequence is $(0,1,3,10,33,109,\ldots)$ and so on.

\section{$k-$Fibonacci sequences and the minimal winning quota in $k-UP$ games}\label{sect:5}

Remarks \ref{rem:21new}, \ref{Rem:32old} as well as \ref{rem:42new} make clear that the problem of finding the sequence of weights (or respectively the sequence of minimal winning quotas) of $k-UP$ games is isomorphic to the one of solving the corresponding constant coefficients $k-$Fibonacci equation.
Then the following results hold:
\begin{teo}\label{teo:51new}
The sequence of all but the top type weights of a $k-UP$ game $G$ with $h$ types is the sequence of the first $h-1$ $k-$Fibonacci numbers; formally:
$$w_t(k)=F_t(k) \quad \text{for } t=1,\ldots,h-1$$
\end{teo}
\begin{proof}
Immediate, by remark \ref{rem:42new} the solution of \eqref{Eqn21bis:third} with initial conditions \eqref{Eqn21bis:first} and \eqref{Eqn21bis:second} is $w_t(k)=F_t(k)$.
\end{proof}
\begin{teo}\label{teo:52new}
The minimal winning quota of a $k-UP$ game $G$ with $h$ types is $q_h(k)=F_h(k)$.
\end{teo}
The proof mimics the one given for theorem \ref{teo:51new} keeping account of formula \eqref{Eqn:37new} and the initial conditions $q_0(k)=0$ and $q_1(k)=1$.

Besides the recursive relation \eqref {Eqn:poly2}, closed form formulae of the $k-$Fibonacci sequences have been given.
\begin{risu}\label{ris:binet}
The Binet's formula (see \cite{FalPla07}, p. 39):
\begin{equation}\label{Enq:binet}
    F_n(k)=\cfrac{[\alpha(k)]^n-[\beta(k)]^n}{\alpha(k)-\beta(k)}
\end{equation}
with $\alpha(k)=(k+(k^2+4)^{1/2})/2$ and $\beta(k)=(k-(k^2+4)^{1/2})/2$
the two solutions of the characteristic equation $x^2-kx-1=0$.
\end{risu}

\begin{risu}\label{ris:floor}
\begin{equation}\label{Enq:floor}
    F_n(k)=)=\sum_{j=0}^{\lfloor(n-1)/2\rfloor} C(n-1-j, j)\cdot k^{n-1-2j}
\end{equation}
with $\lfloor(n-1)/2\rfloor $ the greatest integer contained in $(n-1)/2$, and $C$ binomial coefficients.
\end{risu}

Hence it is:
\begin{equation}\label{Eqn:53k}
\begin{split}
F_1(k)&=1\\
F_2(k)&= k\\
F_3(k)&= k^2+1\\
F_4(k)&= k^3+2k\\
F_5(k)&= k^4+3k^2+1\\
F_6(k)&= k^5+4k^3+3k\\
.....
\end{split}
\end{equation}

Result \ref{ris:floor} gives a polynomial expansion of any number $F_n(k)$ of the $k-$Fibonacci sequence. On the basis of this result, it has been observed (see \cite{FalPla09}, sect. 2.1, table 3) that the sequence of the binomial coefficients of any $F_n(k)$ polynomial representation may be described by a modified Pascal triangle (named 2 Pascal triangle).

We suggest here the alternative equivalent polynomial expansion
of $F_n(k)$:
\begin{risu}\label{ris:floor2}
\begin{equation}\label{Enq:floor2}
    F_n(k)=\sum_{s=0}^{n-1} C^{'}(n,s)\cdot k^{s}
\end{equation}
with $C^{'}(n,s)$ integer coefficients satisfying:
\begin{subequations}\label{Eqn:5star}
\begin{align}
\text{for $n$ odd (even): }& C^{'}(n,0)=1 \text{ (=0)}\\
\text{for $n$ positive integer: }& C^{'}(n,n-1)=1\\
\text{for $n$ positive integer $>1$: }& C^{'}(n,n-2)=0\\
\text{for $n$ positive integer and $0<s<n-2$: }& C^{'}(n,s)=C^{'}(n-2,s)+C^{'}(n-1,s-1)\label{Eqn:5stard}
\end{align}
\end{subequations}
\end{risu}

\begin{rem}
On the basis of such rules the coefficients $C^{'}$ are resumed
by the following modified Pascal triangle, with index
$n=1,2,\ldots$ on the rows and $s=0,1,2,\ldots$ on columns.

\begin{table}[H]
  \begin{center}\label{Tab:Triangle2}\scriptsize
    \begin{tabular}{rr|rrrrrrrrrr}
    \multicolumn{12}{c}{$s$}\\
     & & 0 & 1 & 2 & 3 & 4 & 5 & 6 & 7 & 8 & ...\\
    \hline
    \multirow{7}{0.1cm}{\rotatebox{90}{~\parbox{0.5cm}{$n$}~}} &  1 & 1 &  & & & & \\
       &  2 & 0 & 1 & & & & \\
       &  3 & 1 & 0 & 1 \\
       &  4 & 0 & 2 & 0 & 1\\
       &  5 & 1 & 0 & 3 & 0 & 1\\
       &  6 & 0 & 3 & 0 & 4 & 0 & 1\\
       &  7 & 1 & 0 & 6 & 0 & 5 & 0 & 1\\
       &  8 & 0 & 4 & 0 & 10 & 0 & 6 & 0 & 1\\
       &  ... & ... & ... & ... & ... & ... & ... & ... & ... & ...\\
      \hline
  \end{tabular}
\end{center}
\caption{Modified Pascal Triangle (MPT1) of the coefficients
$C^{'}(n,s)$ of the polynomial representation
$F_n(k)=\sum_{s=0}^{n-1} C^{'}(n,s)\cdot k^{s}$.}
\end{table}
\end{rem}

The rule governing the triangle is very simple: the first
column alternates 1 and 0, the main (external) diagonal has all
elements equal to 1; the diagonal under the main has all
elements equal to 0, all the other internal coefficients follow
the recursive rule \eqref{Eqn:5stard}.

\begin{rem}
As revealed by the notation, the coefficients $C^{'}(n,s)$ do not depend on $k$, i.e. are the same for any $k$.
\end{rem}

\begin{rem}
The sum of the binomial coefficients of the row $n$ (as said
before independent from $k$) is equal to the $n-$th Fibonacci
number; formally: $\sum_{s=0,\ldots,n-1} C^{'}(n,s)=F_n$.
\end{rem}

We will see in section \ref{sect:6}, that this modified Pascal
triangle plays a key role in finding a polynomial expansion of
the minimal winning quota for non $UP$ games.

\section{Fibonacci polynomials and the minimal winning quota of non $UP$ games}\label{sect:6}

In this section we will argue that the minimal winning quota of
a $P$ game not $U$ with $h$ types may be thought as a
generalized version of $F_h(k)$. A key role in this reasoning
is played by the modified Pascal triangle MPT1.

We have been inspired by the following intuition:
\begin{prop}
In any $P$ game not $U$ with $n$ players and $h$ types, the
minimal winning quota $q_h(x_1,x_2,\ldots,x_{h-1})$ should be
given by a polynomial expansion isomorphic to the one found in
the corresponding $k-UP$ game $(k=(n-1)/(h-1))$ with the same
number of players and types.
\end{prop}

More precisely, in such an expansion the number of addends,
obtained as feasible products of exactly $s$ factors chosen
from the free type representation vector, should be given (for
any $s$) by the same coefficient $C^{'}(h,s)$ which multiplies
$k^s$ (in formula \eqref{Enq:floor2}) in the polynomial
expansion of the $k-U$ case.

This way each one of the feasible products of exactly $s$
factors in a $P$ game not $U$ may be thought as the counterpart
of one of the products $k^s$ in a $UP$ game.

Let us comment here a simple example. Consider a $UP$ game $G$
with $n$ players and $h=5$ types with level $k=(n-1)/(h-1)$.
For such a game it is (coherently with the coefficients of row $n=5$ of MPT1 in Table 1):
\begin{equation}\label{Eqn:1pall}
    q_5(k)=F_5(k)=1\cdot k^0+3\cdot k^2+1\cdot k^4=1+kk+kk+kk+kkkk
\end{equation}
For a general $P$ game $G^{'}$ not $U$ with the same number of
players as well as of types of $G$, so that
$$
\sum_{t=1}^{h-1}x_t=n-1
$$
it is (as proved in the next section):
\begin{equation}\label{Eqn:2pall}
    q_5(x_1,x_2,x_3,x_4)=1+x_1\cdot x_2+x_1\cdot x_4+x_3\cdot x_4+x_1\cdot x_2\cdot x_3\cdot x_4
\end{equation}

A comparison between \eqref{Eqn:1pall} and \eqref{Eqn:2pall}
reveals that the three addends $(C^{'}(5,2)=3)$ $kk$ in
\eqref{Eqn:1pall} have been substituted in \eqref{Eqn:2pall} by
the three feasible addends which are the product of exactly two
different factors (chosen among the $x_1$, $x_2$, $x_3$, $x_4$
coherently with the rules in theorem \ref{teo:32}), while the
unique addend $(C^{'}(5,4)=1)$ $kkkk$ has been replaced by the
unique feasible product of four factors; moreover there is also
one addend $(C^{'}(5,0)=1)$, counterpart of $k^0$, which
requires to think that the fictitious product of zero factors
is just equal to one.

Hence the conclusion that $q_5(x_1,x_2,x_3,x_4)$ may be
considered a generalized version of $q_5(k)=F_5(k)$ as well as $q_h(x_1,x_2,\ldots,x_{h-1})$ is a generalized
version of $q_h(k)=F_h(k)$.

It remains to precise the rule governing the choice of the
factors for any feasible combination of $(h,s)$, i.e for which
$C^{'}(h,s)$ is positive, and to show that the expansion
provided by this rule gives exactly the minimal winning quota.

The answer to this couple of problems is given by the following
result which provides the required expansion:
\begin{teo}\label{teo:32}

\begin{subequations}\label{Eqn:37}
\begin{align}
\text{for $h$ even:}\quad q_h(_f\x)&=\sum_{s=1 \hspace{0.1cm}\text{mod}\hspace{0.1cm}2, s<h} \Pi_{j=1,\ldots,s} x_{i_{j}}\label{Eqn37:a}\\
\text{for $h$ odd:}\quad q_h(_f\x)&=1+\sum_{s=0\hspace{0.1cm}\text{mod}\hspace{0.1cm}2, s<h} \Pi_{j=1,\ldots,s} x_{i_{j}}\label{Eqn37:b}
\end{align}
\end{subequations}
sub (in both cases) to the set of constraints on $i_j$:
\begin{subequations}\label{Eqn:37bis}
\begin{align}
i_1&=1 \hspace{0.1cm}\text{mod}\hspace{0.1cm} 2\label{Eqn37bis:a}\\
i_j&<i_{j+1}\qquad\qquad\quad\text{for $j=1,\ldots,s-1$}\label{Eqn37bis:b}\\
i_j&+i_{j+1}=1 \hspace{0.1cm}\text{mod}\hspace{0.1cm} 2 \quad\text{for $j=1,\ldots,s-1$}\label{Eqn37bis:c}\\
i_s&<h    \label{Eqn37bis:d}
\end{align}
\end{subequations}

\end{teo}

\begin{rem}
In \eqref{Eqn37:a} and \eqref{Eqn37:b} all sequences respecting
the constraints are feasible and appear just once (no
repetition of sequences) in the expressions of  $q_h(_f\x)$.
\end{rem}
\begin{rem}
For $h$ even (odd) \eqref{Eqn:37} and \eqref{Eqn:37bis} imply
that $i_1\cdot i_s= 1 \hspace{0.1cm}\text{mod}\hspace{0.1cm}
2$, i.e. $i_s$ is odd  ($i_1\cdot i_s= 0
\hspace{0.1cm}\text{mod}\hspace{0.1cm} 2$, i.e. $i_s$ is even).
\end{rem}
\begin{rem}
Denoting by $F_h$ the $h$-th Fibonacci number and by $n_h$ the
number of addends of $q_h$, it turns out that $n_2=1=F_2$ (by
\eqref{Eqn65:a}) and $n_3=2=F_3$ (by \eqref{Eqn65:b});
moreover the structure of the recursive relation
\eqref{Eqn:37a} makes clear that for any $h>3$,
$n_h=n_{h-1}+n_{h-2}$ so that (by induction on $h$) for any
$h$, $n_h=F_h$
\end{rem}
and the following proposition holds:
\begin{prop}
In any $P$ game with $h$ types the number of addends of the polynomial expansion of $q_h$
is the $h$-th Fibonacci number.
\end{prop}

Hence Theorem \ref{teo:32} says that $q_h$ is given by a sum of
$F_h$ addends. For $h$ even each of these addends is the
product of an odd number $s$, smaller than $h$, of components
of the free type representation $_f\x$ whose sequence of
indices $i_1,i_2,\ldots,i_j,\ldots,i_s$ satisfy constraints of
first index odd \eqref{Eqn37bis:a},  strict monotony
\eqref{Eqn37bis:b}, for $s>1$ alternation of parity
\eqref{Eqn37bis:c} and last index obviously lower than $h$
\eqref{Eqn37bis:d}. For $h$ odd, $s$ is even and the constant 1
appears as one of the $F_h$ addends, while the constraints do
not change.

We list here the first polynomial expansion of $q_h$
analogous to \eqref{Eqn:53k}.
\begin{subequations}\label{Eqn:65}
\begin{align}
q_2&= x_1\label{Eqn65:a}\\
q_3&= 1+x_1x_2\label{Eqn65:b}\\
q_4&= x_1+ x_3+ x_1x_2x_3\label{Eqn65:c}\\
q_5&= 1+ x_1x_2+x_1x_4+ x_3x_4+ x_1x_2x_3x_4\label{Eqn65:d}\\
q_6&=  x_1+ x_3+ x_5 + x_1x_2 x_3+ x_1x_2 x_5+ x_1x_4 x_5+
x_3x_4 x_5 +x_1x_2 x_3 x_4 x_5 \label{Eqn65:e}
\end{align}
\end{subequations}

Now we are going to give the proof of Theorem \ref{teo:32}, but
with a preliminary warning. In the proof we will denote by s.c.
the (standard) constraints \eqref{Eqn37bis:a},
\eqref{Eqn37bis:b} and \eqref{Eqn37bis:c} which survive
unchanged in all formulas, while on the contrary writing
explicitly the updated version of \eqref{Eqn37bis:d} regarding
the constraint on the last index.

\begin{proof}
The theorem clearly holds for $h=2$ as the only feasible $s$
odd, lower than 2, is 1 (and hence $q_2=x_1$), and for $h=3$,
for which the only feasible $s$ even is 2 and hence
$q_3=1+x_1x_2$. Then, we proceed by induction: suppose that the
theorem holds for some values $h-2$ and $h-1$, check that it is
satisfied also for $h$, hence it holds for any positive
integer.\\

Case $h$ even. Let us write\\
\begin{equation}\label{eqn:38}
    q_h= q_{h,1}+q_{h,2}
\end{equation}
with:
\begin{equation}\label{eqn:39}
    q_{h,1}=\sum_{s=1 \hspace{0.1cm}\text{mod}\hspace{0.1cm}2, s<h} \Pi_{j=1,\ldots,s} x_{i_{j}}
\end{equation}
sub to
\begin{subequations}\label{Eqn:39bis}
\begin{align}
s.c.\label{Eqn39:a}\\
i_s=h-1    \label{Eqn39:b}
\end{align}
\end{subequations}
and:
\begin{equation}\label{eqn:310}
    q_{h,2}=\sum_{s=1 \hspace{0.1cm}\text{mod}\hspace{0.1cm}2, s<h-2} \Pi_{j=1,\ldots,s} x_{i_{j}}
\end{equation}
sub to
\begin{subequations}\label{Eqn:310bis}
\begin{align}
s.c.\label{Eqn310:a}\\
i_s<h-2    \label{Eqn310:b}
\end{align}
\end{subequations}

Thus in \eqref{eqn:38} we split the addends of $q_h$ in two
subsets; the first (whose sum is $q_{h,1}$) contains all
addends with last index exactly $h-1$, the second one (whose
sum is  $q_{h,2}$) contains all the other addends with last
index at most $h-3$. The following equalities hold:
\begin{equation}\label{Eqn:311}
    q_{h,1}=x_{h-1}\cdot q_{h-1}
\end{equation}
\begin{equation}\label{Eqn:312}
    q_{h,2}=q_{h-2}
\end{equation}

The \eqref{Eqn:312} follows by definition, the \eqref{Eqn:311}
by:
\begin{equation}\label{Eqn:313}
    \sum_{s=1 \hspace{0.1cm}\text{mod}\hspace{0.1cm}2, s<h} \Pi_{j=1,\ldots,s} x_{i_{j}}=
    x_{h-1}\cdot\Big(1+\sum_{s=0\hspace{0.1cm}\text{mod}\hspace{0.1cm}2, s<h} \Pi_{j=1,\ldots,s} x_{i_{j}}\Big)
\end{equation}
\hspace{3cm} sub to s.c\hspace{5cm} sub to s.c\\

\hspace{2.6cm} $i_s=h-1$ \hspace{5cm} $i_s<h-1$

The left hand side of \eqref{Eqn:313} has already been
explained as the sum of products made by any feasible odd
number of components of the free type representation with last
component exactly $x_{h-1}$ ($h-1$ odd). The right hand side
suggests that all the addends are obtained multiplying by
$x_{h-1}$ all the addends of $q_{h-1}$, i.e. the constant 1
plus the sum of products made by any feasible even number of
components of the representation with last index (even) lower
than $h-1$. This way all constraints of alternation of parity,
last index exactly $h-1$ and hence monotony and first index odd
(either unchanged or $h-1$ for $s=1$) are satisfied. After that
the Theorem (in the even version of the induction) comes
trivially from the following chain:
\begin{equation}\label{Eqn:314}
     q_{h}(_f\x)=q_{h,1}+q_{h,2}=x_{h-1}\cdot q_{h-1}+q_{h-2}=q_h
\end{equation}
which fills the gap between the recursive
solution and the polynomial expansion of the minimal winning quota problem.\\

Case $h$ odd. We still put\\
\begin{equation}\label{eqn:315}
    q_h= q_{h,1}+q_{h,2}
\end{equation}
with:
\begin{equation}\label{eqn:316}
    q_{h,1}=\sum_{s=0\hspace{0.1cm}\text{mod}\hspace{0.1cm}2, s<h} \Pi_{j=1,\ldots,s} x_{i_{j}}
\end{equation}
sub to
\begin{subequations}\label{Eqn:316bis}
\begin{align}
s.c.\label{Eqn316:a}\\
i_s=h-1    \label{Eqn316:b}
\end{align}
\end{subequations}
and:
\begin{equation}\label{eqn:317}
    q_{h,2}=\sum_{s=0 \hspace{0.1cm}\text{mod}\hspace{0.1cm}2, s<h-2} \Pi_{j=1,\ldots,s} x_{i_{j}}
\end{equation}
sub to
\begin{subequations}\label{Eqn:317bis}
\begin{align}
s.c.\label{Eqn317:a}\\
i_s<h-2    \label{Eqn317:b}
\end{align}
\end{subequations}
As before still \eqref{Eqn:311} and \eqref{Eqn:312} hold; in
particular \eqref{Eqn:311} comes from:
\begin{equation}\label{Eqn:318}
    \sum_{s=0 \hspace{0.1cm}\text{mod}\hspace{0.1cm}2, s<h} \Pi_{j=1,\ldots,s} x_{i_{j}}=
    x_{h-1}\cdot\Big(\sum_{s=1\hspace{0.1cm}\text{mod}\hspace{0.1cm}2, s<h-1} \Pi_{j=1,\ldots,s} x_{i_{j}}\Big)
\end{equation}
\hspace{5cm} sub to s.c\hspace{3cm} sub to s.c\\

\hspace{4.6cm} $i_s=h-1$ \hspace{3cm} $i_s<h-1$

The explanation of \eqref{Eqn:318} is (\textit{mutatis
mutandis}) analogous to the one given for \eqref{Eqn:313}.
Hence the chain \eqref{Eqn:314} still holds which completes the
proof.

\end{proof}

It would be useful for the future to denote by  $\varphi_{h,m}$
the $m-$th\footnote{According to whichever order you freely
choose; in the examples of section \ref{sect:6} we will use an
order induced at first by $s$ (in increasing order) and inside
a given $s$, by the lexicographic order dictated by the
sequence of indices (still in increasing order); for $h$ odd
the constant 1 comes first in the ordering.} of the $F_h$
addends of $q_h$ in \eqref{Eqn:37} so that formally:
\begin{equation}\label{Eqn:33}
q_h=\sum_{m=1,\ldots,F_h}\varphi_{h,m}
\end{equation}

With these notation Theorems \ref{teo:31a} and \ref{teo:32} may
be synthesized by:
\begin{equation}\label{Eqn:34}
q_{h+2}=\sum_{m=1,\ldots,F_h}\varphi_{h,m}+x_{h+1}\cdot\sum_{m=1,\ldots,F_{h+1}}\varphi_{h+1,m}
\end{equation}

\section{The symmetry of the function $q_h(_f\x)$}\label{sect:7}

With reference to $h-1$ dimensional vectors, let us introduce
the following definitions.

\begin{defi}\label{def:41}
A couple of vectors $\x=(x_1,x_2,\ldots,x_{h-1})$ and $\y=
(y_1,y_2,\ldots,y_{h-1})$ are each other symmetric if and only
if, for any $t=1,\ldots,h-1$, $x_t=y_{h-t}$.
\end{defi}

\begin{defi}\label{def:42}
A vector function $\Phi$ is symmetric if and only if, for any
couple $\x,\y$ of (each other) symmetric vectors, it is:
$\Phi(\x)=\Phi(\y)$.
\end{defi}

Now, let $\Phi_1,\Phi_2,\ldots,\Phi_m,\ldots,\Phi_n$ be a set
of $n(\geq 2)$ symmetric vector functions, then:

\begin{prop}\label{prop:41}
$\Psi_n=\sum_{m=1,\ldots,n}\Phi_m$ is a symmetric function.
\end{prop}

\begin{proof}
For $n=2$ and $\x,\y$ symmetric vectors it is:
$$\Psi_2(\x)=\Phi_1(\x)+\Phi_2(\x)=\Phi_1(\y)+\Phi_2(\y)=\Psi_2(\y)$$

An induction argument is then applied proving that, if the Prop.
\ref{prop:41} holds for some $m (\geq 2)$, then it holds also
for $m+1$ (and hence for any $n$):
$$\Psi_{m+1}(\x)=\Psi_{m}(\x)+\Phi_{m+1}(\x)=\Psi_{m}(\y)+\Phi_{m+1}(\y)=\Psi_{m+1}(\y)$$
\end{proof}
\begin{rem}
Note that the symmetry of the $\Phi_m$ functions is a
sufficient but not a necessary condition for the symmetry of
their sum $\Psi$.
\end{rem}

Hereafter we will apply these ideas to the function
$q_h(_f\x)$ defined in \eqref{Eqn:37}. Let us start with the
following

\begin{prop}\label{prop:42}
Suppose an addend $\varphi_{h,m}$ of $q_{h}(_f\x)$ is not
symmetric, then among the other addends there is a non
symmetric one, denoted hereafter by $\varphi_{h,\overline{m}}$
such that $\Phi_{h,m}=\varphi_{h,m}+\varphi_{h,\overline{m}}$
is a symmetric function. Recalling that
$\varphi_{h,m}=\Pi_{j=1,\ldots,s} x_{i_{j}}$, this happens when
$\varphi_{h,\overline{m}}=\Pi_{j=s,\ldots,1} x_{{h-i}_{j}}$.
\end{prop}

\begin{proof}
First of all we must check that the factors of
$\varphi_{h,\overline{m}}$ satisfy the feasibility conditions
(monotony, alternate parity, first odd and last lower than $h$)
of Theorem \ref{teo:32}.

Monotony and alternate parity of the sequence of
$\varphi_{h,\overline{m}}$ indices are a straightforward
consequence of monotony and alternate parity of the sequence of
$\varphi_{h,m}$ indices; for $h$ even $i_s$ is odd, which
implies that the first index $h-i_s$ of
$\varphi_{h,\overline{m}}$ is odd too; on the contrary, for $h$
odd $i_s$ is even and $h-i_s$ is odd too and the ``first odd''
condition is satisfied; finally $i_j$ positive and lower than
$h$ for any $j$ in $\varphi_{h,m}$ imply the same property in
$\varphi_{h,\overline{m}}$.

After that, let us consider any couple $\x,\y$ of symmetric
vectors according to definition \ref{def:41}; for such vectors
it is immediately seen that:
$$\varphi_{h,m}(\x)=\varphi_{h,\overline{m}}(\y)$$
and
$$\varphi_{h,\overline{m}}(\x)=\varphi_{h,m}(\y)$$
then
$$
\Phi_{h,m}(\x)=\varphi_{h,m}(\x)+\varphi_{h,\overline{m}}(\x)=\varphi_{h,\overline{m}}(\y)+\varphi_{h,m}(\y)=\Phi_{h,m}(\y)
$$
\end{proof}

It is now easy to show that:
\begin{teo}\label{teo:41}
For any $P$ game with $h$ types, $q_h(_f\x)$ is a
symmetric function
of the free type
representation of the game.
\end{teo}
\begin{proof}
At first, keep account of the fact that for a given $h$ only
some of the $\varphi_{h,m}$ are symmetric functions, while some
other are not symmetric. Symmetry holds if $\varphi_{h,m}=1$ or
if, for any $t=1,\ldots,h-1$, both or none of the components
$x_t$ and $x_{h-t}$ are factors of $\varphi_{h,m}$.

On the other side for the $\varphi_{h,m}$ which are not
symmetric functions, Prop. \ref{prop:42} holds and this grants
that $q_h(_f\x)$ may be written as a sum of symmetric functions
(the $\varphi_{h,m}$ already symmetric and the $\Phi_{h,m}$
symmetric by Prop. \ref{prop:42}); hence, by Prop.
\ref{prop:41}, it is a symmetric function.

\end{proof}

\section{The equality of the winning quota in symmetric
games}\label{sect:8}

Let us give the following definition:

\begin{defi}
Two $P$ games $G$ and $\overline{G}$ whose free type
representations $_f\x$ and $_f\overline{\x}$ are (each other)
symmetric and not identical are twin games (twins).
\end{defi}

\begin{rem}
For twins $h=\overline{h}$: twins have the same number of types
and of course of players.
\end{rem}

The following theorem concerning twins turns out to be a
straightforward corollary of the  results given in section
\ref{sect:4}:

\begin{teo}\label{teo:51}
Let $G$ and $\overline{G}$ be twins with minimal winning quota
$q$ and respectively $\overline{q}$ bar; then it is
$q=\overline{q}$.
\end{teo}

\begin{proof}
The symmetry of the function $q_h(_f\x)$ on one side and of
the vectors $_f\x$ and $_f\overline{\x}$ on the other imply
$$q_h(_f\x)=q_h(_f\overline{\x})$$

Hence any couple of twin $P$ games has the same minimal winning
quota.
\end{proof}

\begin{rem}
Elsewhere (see \cite{PPZ2}, sect. 4) we give an alternative
proof of Theorem \ref{teo:51} based on the transposition
properties of the incidence matrices of any couple of twin $P$
games. While probably less elegant from a purely mathematical
point of view, the proof given in this paper has the advantage
to give an enlightening analytic explanation of the true
reasons behind the equality of the minimal winning quotas of
twin games.
\end{rem}

\section{Examples}\label{sect:9}

\paragraph{Ex. 1}
Let us consider the nine person $P$ game $G$ with free type
representation $_f\x=(3,1,2,2)$ and minimal homogeneous
representation $(26;1,1,1,3,4,4,11,11,15)$. There are 5 types
in $G$ $(h-1=4)$ and thus $F_5=5$ addends in the polynomial expansion
$q_5(_f\x)$:
$$q_5(_f\x)=1+ x_1x_2+ x_1x_4+ x_3x_4+ x_1x_2x_3x_4
=1+3+6+4+12=26$$

Three of the addends: 1, $x_1x_4$ and $x_1x_2x_3x_4$ are
symmetric functions of a four dimension vector; the sum
$x_1x_2+x_3x_4$ of the remaining (non symmetric) two is
symmetric too.

The twin $\overline{G}$ has free type representation
$_f\overline{\x}=(2,2,1,3)$ and minimal homogeneous
representation $(26;1,1,2,2,5,7,7,7,19)$. Now:
$$
q_5(_f\overline{\x})=1+x_1x_2+x_1x_4+x_3x_4+x_1x_2x_3x_4=1+4+6+3+12=26$$

Note that the first, third and fifth addend are the same found
in the expansion of $q_5(_f\x)$, while the second and the
fourth have been inverted (so that their sum 7 did not change).

\paragraph{Ex. 2} Let us consider the set of $k-UP$ games $G$ with six types and free type representation $_f\x=(k,k,k,k,k)$. Coherently with row corresponding
to $n=6$ of the triangle in Table 1 (here $h=6$), the minimal
winning quota is expressed in the $k-$Fibonacci polynomial
version by $q_6(\textbf{k})=3k+4k^3+k^5$,
which is the sixth (positive) number of the $k-$Fibonacci
sequence. Note that the sum of the coefficients is 8, i.e. the
sixth Fibonacci number. In particular for $k=2$ the minimal homogeneous representation of the 2-$UP$ game is
(70;1,1,2,2,5,5,12,12,29,29,41). Then, the minimal winning quota
$q_6(\textbf{2})=3\cdot2+4\cdot2^3+2^5=70$, is just the sixth positive number of the Pell
sequence (1,2,5,12,29,70). Hence the amazing equality
between the type weights and the Pell sequence and, more generally, between the type weights and
the $k-$Fibonacci sequences in $k-UP$ games.

\paragraph{Ex. 3}

Let us consider the twelve person $P$ game $G$ with free type
representation $_f\x=(2,1,3,2,2)$ and minimal homogeneous
representation $(61;1,1,2,3,3,3,11,11,25,25,36)$. There are
6 types in $G$ $(h-1=5)$ and thus $F_6=8$ addends in the
polynomial expansion $q_6(_f\x)$:
\begin{equation*}
\begin{split}
q_6(_f\x)&=x_1+x_3+x_5+x_1x_2x_3+x_1x_2x_5+x_1x_4x_5+x_3x_4x_5+
         +x_1x_2x_3x_4x_5\\
         &=2+3+2+6+4+8+12+24=61
\end{split}
\end{equation*}

Among the addends, only the second $x_3$ $(3)$ and the last one $x_1x_2x_3x_4x_5$
$(24)$ are symmetric functions of a five
dimension vector, the other six may be coupled in three pairs,
the first and the third, the fourth and the last but one, the
fifth and the sixth, so as each pair has symmetric sum $(x_1+
x_5= 2+2=4)$, $(x_1x_2x_3+x_3x_4x_5=6+12=18)$,
$(x_1x_2x_5+x_1x_4x_5= 4+8=12)$.

The twin $\overline{G}$ has free type representation
$_f\overline{\x}=(2,2,3,1,2)$ and minimal homogeneous
representation $(61;1,1,2,2,5,5,5,17,22,22,39)$.
\begin{equation*}
\begin{split}
q_6(_f\overline{\x})&=x_1+x_3+x_5+x_1x_2x_3+x_1x_2x_5+x_1x_4x_5+x_3x_4x_5+
         +x_1x_2x_3x_4x_5\\
         &=2+3+2+12+8+4+6+24=61
\end{split}
\end{equation*}
The second and the last one addend are the same found in the
expansion of $q_6(_f\x)$, the other pairs have been
interchanged according to the rule of symmetry of the sum.

We underline that the two games $G$ and $\overline{G}$ of Ex. 3 are two of the many non $UP$ games with 11 players and 6 types, which may be thought as counterpart of the $2-UP$ game of Ex. 2 (with 11 players and 6 types too); on the other side, the minimal winning quota polynomial expansions of $G$ and respectively $\overline{G}$ may be considered as two generalized versions of the polynomial expansion of the sixth number of the Pell sequence.

\section{Conclusions}\label{sect:10}

In the first part of the paper we give a closed form solution in terms of a polynomial expansion to the minimal winning quota of parsimonious games.
Preliminarily, we underline that a parsimonious game is unequivocally described by its free type representation vector and introduce the class of $k-$uniform parsimonious games, i.e. games whose free type representation is uniform at the level $k$.
We show that the minimal winning quota of parsimonious games satisfies a second order linear, homogeneous, finite difference equation, with nonconstant coefficients, except for the uniform case, and index $h$, the number of players types.
We check that, keeping account of the initial conditions, the solution of the equation for the k uniform case is just the $k-$Fibonacci sequence.
Then we recall that any element of the $k-$Fibonacci sequence may be expressed through a polynomial expansion in $k$, whose coefficients are linked to the binomial coefficients of the Pascal triangle.
In order to build a bridge towards the solution of the equation in the nonconstant coefficients case, we introduce an alternative, perfectly equivalent, polynomial expansion, whose coefficients are given by a conveniently modified Pascal triangle.
After that we show that the minimal winning quota in non uniform games  (nonconstant coefficient case) is given by a polynomial expansion of the free type representation vector, which may be considered as the generalized version of the polynomial expansion of the $k-$Fibonacci sequence for the corresponding $k-$uniform games (constant coefficient case).
In the second part of the paper we demonstrate that the minimal winning quota is a symmetric function of its argument (the free type representation vector).
Exploiting this property it is immediate to prove that twin parsimonious games, i.e. couple of games whose type representations are each other symmetric, share the same minimal winning quota.
Finally we underline that this offers an alternative, enlightening proof of a result which has been obtained elsewhere, exploiting properties of the determinants of the incidence matrices of twin games.


\end{document}